\documentclass[a4paper]{amsart}
\usepackage{amssymb}
\usepackage{mathrsfs}
\newtheorem{theorem}{Theorem}[section]
\newtheorem{lemma}[theorem]{Lemma}

\newtheorem{prop}[theorem]{Proposition}
\theoremstyle{definition}
\newtheorem{definition}[theorem]{Definition}
\newtheorem{example}[theorem]{Example}

\theoremstyle{remark}
\newtheorem{remark}[theorem]{Remark}

\begin{document}
\title[Divergent integrals and residues]
{Divergent integrals, residues of Dolbeault forms,
and asymptotic Riemann mappings}
\author{Giovanni Felder}
\address{Department of mathematics,
ETH Zurich, 8092 Zurich, Switzerland}
\email{felder@math.ethz.ch}
\author{David Kazhdan}
\address{Einstein Institute of Mathematics,
The Hebrew University of Jerusalem,
Jerusalem 91904, Israel}
\email{kazhdan@math.huji.ac.il}

\begin{abstract}
  We describe the asymptotic behaviour and the dependence on the
  regularization of logarithmically divergent integrals of products of
  meromorphic and antimeromorphic forms on complex manifolds. Our
  formula is expressed in terms of residues of Dolbeault forms, a
  notion introduced in this paper.  The proof is based on a result on
  the asymptotic behaviour of Riemann mappings of small domains.
\end{abstract}
\maketitle
\section{Introduction}\label{s-1}
Let $X$ be a compact complex manifold of dimension $n$ and
$D\subset X$ a smooth hypersurface. We are interested in
regularizations of divergent integrals of the form
\[
  \frac1{2\pi i}\int_X\alpha\wedge\bar\beta,\quad \alpha\in
  \Gamma(X,\Omega^n_X(mD)),\quad\beta\in\Gamma(X,\Omega^n_X(D)).
\]
Here $\Gamma(X,\Omega^p_X(mD))$ denotes the space of meromorphic
$p$-forms on $X$  whose poles are on $D$ and are of order at most
$m$.  Integrals of this type arise in perturbative superstring theory,
see \cite{Witten}: the terms in the perturbation series for
superstring amplitudes are conditionally convergent improper integrals
of products of holomorphic times antiholomorphic differential forms on
moduli spaces of supercurves.  They thus depends on the way one takes
the limit and it is important to understand this dependence. We
consider here the basic case of an integral of this kind, with only
one divisor component.

The way to make sense of integrals of the form above is to introduce a
regularization and study the asymptotic behaviour as the
regularization is removed. Specifically the regularization consists of
replacing the integral over $X$ by the integral over the complement of
a suitable family of tubular neighbourhoods $N_\epsilon(D)$ of $D$
shrinking to zero size as the regularization parameter $\epsilon>0$
tends to zero. It turns out that the integral behaves as
$I_0\log\epsilon+I_1+O(\epsilon)$ as $\epsilon\to 0$, with $I_0$
independent of the choice of family of neighbourhoods.  In
applications to superstring theory the integral is ``finite'' in the
sense that $I_0$ actually vanishes, so that one can define the
integral as $I_1$.  However $I_1$ depends on the choice of
regularization and it is important to understand this dependence. The
aim of this paper is to describe the divergent term and dependence on
the choice of regularization, generalizing calculations of
\cite{Witten}, who considered the case $m=1$. More precisely, we
consider regularizations defined by a {\em cut-off function}
$\lambda$, a smooth positive function on $X\smallsetminus D$ such that
$\lambda/|f|$ extends smoothly to $D$ for any local equation $f=0$
defining $D$. For example, one can take $\lambda$ to be the geodesic
distance to $D$ for some Hermitian metric on $X$. Then for small
$\epsilon>0$, $N_\epsilon(D)=\{x\in X:\lambda(x)<\epsilon\}$ is a
tubular neighbourhood of $D$ and we define the regularized integral as
the integral over its complement.

The first result is that the divergent part can be written in terms of
residues
\begin{equation}\label{e-1}
  \frac {1} {2\pi i} \int_{\lambda\geq\epsilon}
  \alpha\wedge\bar\beta = 2(-1)^{n-1}\log\,\epsilon \int_D\mathrm{Res}\,\alpha
  \wedge
  \overline{\mathrm{Res}\,\beta}+I_1(\lambda)+O(\epsilon),\quad
  (\epsilon\to0).
\end{equation}
Here $\mathrm{Res}$ is the Leray residue, sending de Rham cohomology
classes on $X\smallsetminus D$ to de Rham cohomology classes on
$D$. For meromorphic closed forms with first order poles on $D$, such
as $\beta$, it reduces to the Poincar\'e residue, which is defined as
a holomorphic differential form on $D$, not just a cohomology class.
It follows in particular that the coefficient of $\log\epsilon$ only depends
on the de Rham cohomology classes of the closed forms $\alpha$ and $\beta$. 

To describe the dependence of the finite part $I_1(\lambda)$ on the
choice of $\lambda$ we notice that any two cut-off functions differ
by multiplication by $\exp\varphi$ for some smooth real-valued function
$\varphi$. We find:
\begin{equation}\label{e-2}
    I_1(e^\varphi\lambda)=I_1(\lambda)- 2(-1)^{n-1}\int_D
    \mathrm{Res}_\partial(\varphi\alpha)\wedge\overline{\mathrm{Res}\,\beta},
\end{equation}
see Theorem \ref{t-3}. 
Here $\mathrm{Res}_\partial$ is the residue of a Dolbeault form, a notion
we introduce in this paper. It maps smooth forms $\omega$ of type $(n,0)$
on $X\smallsetminus D$ with poles on $D$ to classes of forms
on $D$ of type $(n-1,0)$ on $D$ modulo the image of $\partial$ (it is 
more convenient in this setting to consider $\partial$ as the Dolbeault
differential; it is related to the usual convention by complex conjugation).

A remarkable feature of \eqref{e-2} is that the variation of the integral is
linear in $\varphi$ and is independent of $\lambda$.

Our formulas are a special case of a more general formula where
one takes $\alpha$ to be a not necessarily closed
$C^\infty$ form of type $(n,0)$ on $X\smallsetminus D$
and whose support has compact closure in $X$ (we don't need $X$ to be
compact in this case). The formulas \eqref{e-1}, \eqref{e-2} hold except
that one has to replace $\mathrm{Res}\,\alpha$ in \eqref{e-1} by
$\mathrm{Res}_\partial\alpha$. This more general formula can be checked
locally via a partition of unity and the formula for meromorphic $\alpha$
follows as a corollary.

The proof of these statements is based on a result on the asymptotic behaviour
of Riemann mappings for small domains, which might be of independent
interest. Let $X=\mathbb C$ and $D=\{0\}$ and let $\lambda$ be a
cut-off function such that $\lambda^2$ is real analytic. Then, for
small $\epsilon>0$, $\lambda(z)<\epsilon$ defines a simply connected
domain containing the origin, so, by the Riemann mapping theorem, there
is a unique biholomorphic map $z\mapsto w=f_\epsilon(z)$ onto the disk
$|w|<\epsilon$ sending $0$ to $0$ and such that $f_\epsilon'(0)>0$. We
show (see Theorem \ref{t-1}) that the family of maps $f_\epsilon(z)$
converges as $\epsilon\to 0$ to a univalent map defined in a
neighbourhood of the origin and is in fact the restriction of a
holomorphic function of two variables $\epsilon,z$ defined on a
polydisk around
$(0,0)$. In Theorem \ref{t-2} we give an explicit formula for
the limiting map.

This construction is used to prove \eqref{e-1}, \eqref{e-2} by 
first reducing to a local calculation and then reducing
to the case where $\lambda=|f|$ for a local equation
$f=0$ of $D$.

The paper is organized as follows. In Section \ref{s-2} we introduce a
notion of residue for $C^\infty$ differential forms of type $(p,0)$ on
a complex manifold with poles on a smooth hypersurface. The divergent
part of the integral and the dependence on the cut-off of the finite
part are expressed in terms of this residue in Section \ref{s-3}.
Our result on the asymptotic behaviour of Riemann mappings for small domains 
is proved in Section \ref{s-4}. Section \ref{s-5} contains the proof
of the formulas for divergent integrals presented in Section \ref{s-3}

\subsection*{Acknowledgments}
We thank Curtis McMullen for discussion and suggestions on the topic of
Section \ref{s-4}.
G.F. was supported in part by the National Centre of Competence in
Research ``SwissMAP --- The Mathematics of Physics'' of the Swiss
National Science Foundation.

\section{Residues of Dolbeault forms}\label{s-2}
Let $X$ be an $n$-dimensional complex manifold with sheaf
$\mathcal O_X$ of germs of holomorphic functions and let $D\subset X$
be a smooth hypersurface. Thus every point of $D$ has an open
neighbourhood $U$ so that the ideal of $\mathcal O_X$ of functions
vanishing on $D\cap U$ is generated by a function $f\in\mathcal O_X(U)$.
We call $f$ a {\em local equation} for $D$. Any two local equations differ
by multiplication by a nowhere vanishing holomorphic function.

\subsection{Residues of smooth differential forms}
A $C^\infty$ differential form $\alpha$ on $X\smallsetminus D$ is said
to have a pole of order $\leq m$ if around every point of $D$, $f^m\alpha$
extends smoothly to $X$ for a local equation $f$ for $D$. If $\alpha$ 
is closed and has a pole of order $\leq 1$ then locally 
$\alpha=df/f\wedge\sigma+\tau$ for some smooth forms $\sigma,\tau$. Then
$\sigma|_D$ is closed, globally defined, and independent of choices; it
is called {\em Poincar\'e residue} $\mathrm{Res}\,\alpha$,
see \cite{Leray}, p.~83. If $\alpha$ is a holomorphic closed form then 
$\mathrm{Res}\,\alpha$ is also holomorphic.
The {\em Leray residue} of a closed differential
form $\alpha$ on $X\smallsetminus D$ is a de Rham cohomology class
on $D$. It is defined to be the Poincar\'e 
residue for forms with first order pole and in general it is the
residue of any form with first order pole in the de Rham cohomology
class of $\alpha$, \cite{Leray}, Th\'eor\`eme 1. These notions were
further developed in \cite{Griffiths, Deligne}. In this
paper we develop part of the theory for the (conjugate) Dolbeault
complex replacing the de Rham complex.  

\subsection{Dolbeault forms}
Let $C^\infty_X$ be the sheaf of germs of smooth functions on $X$.
Let $\mathcal A_X^{r}=\oplus_{p+q=r} \mathcal A_X^{p,q}$ be the
decomposition of the $C^\infty_X$-module of smooth differential
$r$-forms into forms of type $(p,q)$. The de Rham differential on
$\mathcal A_X^{\bullet}=\oplus_{r=0}^{2n} \mathcal A_X^r$ decomposes
as $d=\partial+\bar\partial$ with
\[
  \partial(\mathcal A_X^{p,q})\subset \mathcal A_X^{p+1,q},\quad
  \bar\partial(\mathcal A_X^{p,q})\subset \mathcal A_X^{p,q+1}.
\]
Let $\mathcal A_X^{p,q}(*D)$ be the sheaf of smooth $(p,q)$-forms
$\alpha$ on $X\smallsetminus D$ so that $f^m\alpha$ extends to $X$ for
some $m$ and local equation $f$ for
$D$.

Similarly, let $\mathcal A_X^{p,q}\langle D\rangle$ be the sheaf
of smooth logarithmic forms. By definition, sections of
$\mathcal A_X^{p,q}\langle D\rangle$ on an open set $U\subset X$ are
smooth $(p,q)$ forms $\alpha$ on $U\smallsetminus (U\cap D)$ such that
for any local equation $f$, both $f\alpha$ and $df\wedge\alpha$ are
regular on $U$. Both $\mathcal A_X^{p,q}(*D)$ and
$\mathcal A_X^{p,q}\langle D\rangle$ are double complexes with
differentials $\partial,\bar\partial$. The de Rham sheaves of
meromorphic forms $\Omega_X^\bullet(*D)$,
$\Omega_X^\bullet\langle D\rangle$ are subcomplexes of
$(\mathcal A_X^{\bullet,0}(*D),\partial)$,
$(\mathcal A_X^{\bullet,0}\langle D\rangle,\partial)$ respectively.

\begin{prop}\label{p-1}
  The inclusion map
  \[
  (\mathcal A_X^{\bullet,0}\langle D\rangle,\partial)\to (\mathcal
  A_X^{\bullet,0}(*D),\partial)
  \]
  is a quasi-isomorphism. The same holds if we take sections with
  compact support.
\end{prop}
\begin{proof} Let $F_j^\bullet$ be the subsheaf 
of $\mathcal A^{\bullet,0}_X(*D)$
consisting of sections $\alpha$ such that $f^{j+1}\alpha$ and 
$f^{j}df\wedge\alpha$
extend smoothly to $D$ for any local equation $f$. These subsheaves
form an increasing filtration
\[
\mathcal A_X^{\bullet,0}\langle D\rangle=F^\bullet_0\subset
F^\bullet_1\subset\cdots\subset\mathcal A^{\bullet,0}_X(*D)=\cup_jF_j.
\]
of subcomplexes of $\bar{\mathcal O}_X$-modules. Locally, the
associated graded $F_{j}/F_{j-1}$ consists of classes of sections
of the form
\[
\alpha =\frac{df}{\ f^{j+1}}\wedge\sigma+\frac{1}{f^j}\tau,
\]
where $\sigma$ and $\tau$ are defined modulo the ideal generated by
$f$ and $df$.  The differential sends
$(\sigma,\tau)$ to $(-\partial\sigma-j\tau,\partial\tau)$.
For $j\geq 1$, this implies that if $\alpha$ is a closed form
then $\tau\equiv-\partial\sigma/j\mod f,df$ and thus 
\[
\alpha=\partial\left(-\frac1{jf^j}\sigma\right)
\]
in $F_j/F_{j-1}$. Thus the cohomology of the associated graded 
is trivial except for $j=0$, where $F_0/F_{-1}=F_0$ is the subcomplex
of logarithmic forms. The same argument works for the subcomplex
of sections with compact support.
\end{proof}

\begin{prop}\label{p-Poincare}
  The Poincar\'e residue map on $(n,0)$ forms uniquely extends to a
  $C^\infty_X$-linear map of complexes of $\bar{\mathcal O}_X$-modules
  \[
    \mathrm{Res}\colon (\mathcal A^{\bullet,0}_X\langle
    D\rangle/\mathcal A^{\bullet,0}_X, \partial)\to (i_*\mathcal
    A^{\bullet,0}_D[-1],\partial)
  \]
  It restricts to a $C^\infty_X$-linear map of complexes of
  $\bar{\mathcal O}_X$-modules
  \[
    \mathrm{Res}\colon ({}^c\!\mathcal A^{\bullet,0}_X\langle
    D\rangle/{}^c\!\mathcal A^{\bullet,0}_X, \partial)\to (i_*{}^c\!\mathcal
    A^{\bullet,0}_D[-1],\partial)
  \]
  on the subcomplex of forms with compact support.
\end{prop}
\begin{proof} 
  Logarithmic forms are locally of the form $\alpha= df/f\wedge\sigma$
  where $\sigma\in\mathcal A^{\bullet,0}_X$ and $f$ is a holomorphic
  function generating the ideal of $D$.  We define the residue as
  $\mathrm{Res}\,\alpha=\sigma|_D$. As in the classical case the
  residue is independent of the choice of $f$ and $\sigma$. By
  construction it is a map of $C^\infty_X$-modules and, since
  $\partial\alpha=- df/f\wedge\partial\sigma$, it commutes with the
  differential (taking into account that the shift functor $[-1]$
  involves a change of sign of the differential). It is clear that
  $\mathcal A^{\bullet,0}_X$ lies in the kernel of the residue
  map. This proves existence. As for uniqueness, let $z_1,\dots,z_n$
  be local coordinates on an open set $U$ such that $D\cap U$ is
  defined by $z_1=0$.  $\mathcal A^{\bullet,0}_X\langle D\rangle(U)$
  is an algebra over $C^\infty_X(U)$ generated by
  $dz_1/z_1,dz_2,\dots dz_n$ for which the residue is the classical
  Poincar\'e residue and uniqueness follows from
  $C^\infty_X$-linearity. It is clear that the residue maps forms with
  compact support to forms with compact support.
\end{proof}

Combining Prop.~\ref{p-1} with \ref{p-Poincare} we get maps of
complexes (with differential $\partial$)
\[
\mathcal A_X^{\bullet,0}(*D)\longleftarrow \mathcal A_X^{\bullet,0}\langle D\rangle
\stackrel{\mathrm{Res}}\longrightarrow i_*\mathcal A_D^{\bullet,0}[-1],
\]
where the first arrow is a quasi-isomorphism.

Since these sheaves are fine and thus acyclic for the global section functor,
we obtain morphisms of complexes 
\[
\Gamma(X,\mathcal A_X^{\bullet,0}(*D))\longleftarrow \Gamma(X,\mathcal A_X^{\bullet,0}\langle D\rangle)
\stackrel{\mathrm{Res}}\longrightarrow \Gamma(D,\mathcal A_D^{\bullet,0}[-1]),
\]
and the first arrow is still a quasi-isomorphism. Passing to cohomology
we obtain a map
\[
  \mathrm{Res}_\partial\colon H^\bullet(\Gamma(X,\mathcal
  A_X^{\bullet,0}(*D)),\partial)\to H^{\bullet}(\Gamma(D,\mathcal
  A_D^{\bullet,0}),\partial)[-1]
\]
 Note that since
$(\mathcal A_D^{\bullet,0},\partial)$ is, by Dolbeault's theorem, a
resolution of the sheaf $\bar{\mathcal O}_D$ of antiholomorphic
functions and is acyclic for the functor of global sections, we
have:
\[
  H^{\bullet}(\Gamma(D,\mathcal A_D^{\bullet,0}),\partial) \cong
  H^{\bullet}(D,\bar {\mathcal{O}}_D).
\]
We are mostly concerned with the case of top forms
$\mathcal A_X^{n,0}(*D)$, which are automatically
$\partial$-closed. 
\begin{definition}
  The {\em residue of top Dolbeault forms} $\mathrm{Res}_\partial$ is
  the composition
\[
  \mathrm{Res}_\partial\colon \Gamma(X,\mathcal A_X^{n,0}(*D)) \to
  \Gamma(D,\mathcal A^{n-1,0}_D)/\mathrm{Im}(\partial)\cong
  H^{n-1}(D,\bar{\mathcal O}_D)
\]
\end{definition} 
 
By construction the residue has the following properties:
\begin{prop}\label{p-properties}
\ 
\begin{enumerate}
\item[(i)] $\mathrm{Res}_\partial(\partial\gamma)=0$ for all $\gamma\in\Gamma(X,\mathcal A^{n-1,0}_X(*D))$.
\item[(ii)] $\mathrm{Res}_\partial$ vanishes on forms extending to smooth
forms on $X$.
\item[(iii)] The support of $\mathrm{Res}_\partial(\alpha)$ is contained
in $D\cap\mathrm{supp}(\alpha)$.
\item[(iv)] $\mathrm{Res}_\partial$ coincides with the Poincar\'e residue
on $(n,0)$-forms with simple pole on $D$.
\end{enumerate}
\end{prop}
\begin{example}\label{e-2.1}
Let $n=1$, $D=\{p\}$ a point, $z=x+iy$ 
a coordinate vanishing at $p$ and $m\geq1$. 
Then a local section $\alpha\in \mathcal A^{1,0}_X(mD)$ 
has the form $g(x,y)dz/z^{m}$ for some smooth function $g$ defined in a
neighbourhood of $(0,0)$. Then
\[
\mathrm{Res}_\partial\,\alpha=\frac1{(m-1)!}\,\partial_z^{m-1}g(0,0),\qquad
\partial_z=\frac12(\partial_x-i\partial_y). 
\]
Indeed, this holds by definition for $m=1$. The general case follows
by induction from 
\[
  g\frac {dz}{z^{j+1}} =
  \partial_zg\frac {dz}{jz^{j}}-\partial_z\left(\frac{g}{jz^j}\right)dz \equiv
  \partial_zg\frac {dz}{jz^{j}} \mod \mathrm{Im}\,\partial,\qquad j\geq1.
\]
\end{example}
\subsection{Comparison with the Leray residue}
There is a canonical map 
\[
\psi\colon H^p(D,\mathbb C)\to H^{p}(\Gamma(D,\mathcal A^{\bullet,0}_D),\partial)
\]
 from the de Rham cohomology: it sends the class of
a closed $p$-form $\alpha$ with Hodge decomposition
\[
\alpha=\alpha^{p,0}+\alpha^{p-1,1}+\cdots+\alpha^{0,p}
\]
to the class of its top component $\alpha^{p,0}$.
\begin{prop}\label{p-2.6}
$\mathrm{Res}_\partial(\alpha)=\psi(\mathrm{Res}\,\alpha)$ for all
$\alpha\in \Gamma(X,\Omega_X^n(*D))$.
\end{prop}
\begin{proof}
If $\alpha$ is a differential form on $X\smallsetminus D$ then its de Rham
cohomology
class contains a form $\tilde\alpha$ with first order pole and the residue
is the residue of $\tilde\alpha$. Now suppose that $\alpha$ is of type
$(n,0)$ with pole on $D$. Then 
\[
\alpha=\tilde\alpha+d\gamma=\tilde\alpha +\partial\gamma+\bar\partial\gamma
\equiv\tilde\alpha+\partial \gamma \mod \bar F^1\mathcal A^\bullet(*D),
\]
where $\bar F^1\mathcal A^p(*D)=\oplus_{q\geq1}\mathcal A^{p-q,q}(*D)$.
Then $\mathrm{Res}\,\alpha=\sigma|_D$, where $\tilde\alpha=df/f\wedge\sigma+\tau$. On the other hand if we denote by $\alpha_p$ the $(p,0)$ part of a
$p$-form $\alpha$,
\[
\alpha=\tilde\alpha_n+\partial\gamma_{n-1}.
\]
Then $\alpha_n=df/f \wedge\sigma_{n-1}$ and by definition 
$\mathrm{Res}_\partial \alpha=\sigma_{n-1}|_D=(\mathrm{Res}\,\alpha)_{n-1}$.  
\end{proof}
%%%%%%%%%%%%%%%%%%%%%%%%%%%%%%%%%%%%%%%%%%%%%%%%%%%%%%%%%%%%%%%%%%%%%%%%%%%%%%
\section{Divergent integrals}\label{s-3}
The divergent integrals we are interested in are regularizations of integrals
of the form $\int_X\alpha\wedge\bar\beta$ where $\alpha$, $\beta$ are meromorphic differential
$n$-forms on $X$ with poles on $D$ and $\beta$ has a simple pole. The regularization consists
on integrating on the complement of a small tubular neighborhood of radius $\epsilon$ of $D$ and subtracting the
divergent term as $\epsilon\to 0$. The result depends on the regularization and
we want to describe this dependence. 

\begin{definition}
  A {\em cut-off function} is a smooth function on $X$ with values in
  $[0,\infty)$ so that
  \begin{enumerate}
  \item[(i)] $\lambda(x)>0$ for all $x\in X\smallsetminus D$.
  \item[(ii)] For any local equation
    $f\colon U\to \mathbb C$ for $D$, $\lambda/|f|$ extends to a smooth
    positive function on $U$.
  \end{enumerate}
\end{definition}
Cut-off functions exist as they can be constructed by patching local
functions $\lambda(x)=|f(x)|$ using a partition of unity. They form a
torsor over positive functions, since the ratio of two cut-off functions
is a smooth positive function on $X$.

\begin{theorem}\label{t-3} 
  Let $D\subset X$ be a smooth hypersurface in a compact complex
  manifold.  Let $\lambda\in C^\infty(X)$ be a cut-off function and
  $m\in\{0,1,2,\dots\}$.  Then for any
  $\alpha\in \Gamma(X,\Omega_X(mD))$, $\beta\in\Gamma(X,\Omega_X(D))$,
  \[
    \frac1{2\pi i} \int_{\lambda\geq\epsilon}\alpha\wedge\bar\beta
    =I_0\log\epsilon+I_1(\lambda)+O(\epsilon)
  \]
  where
  \[
    I_0=2(-1)^{n-1}\int_D
    \mathrm{Res}\,\alpha\wedge\overline{\mathrm{Res}\,\beta},
  \]
  and for $\varphi\in C^\infty(X,\mathbb R)$,
  \[
    I_1(e^\varphi\lambda)=I_1(\lambda)- 2(-1)^{n-1}\int_D
    \mathrm{Res}_\partial(\varphi\alpha)\wedge\overline{\mathrm{Res}\,\beta},
  \]
\end{theorem}
The integrals over $D$ in this theorem are understood as integrals of
representatives in the equivalence classes of the residues: in $I_0$ the
Leray residues are de Rham cohomology classes (while Res$\,\beta$ is defined
as a closed form) and the integral is independent of choices by Stokes's
theorem. In the second integral over $D$, $\mathrm{Res}_\partial$ is defined
up to addition of a $\partial$-exact form. Since $\overline{\mathrm{Res}\,\beta}$ is a closed form of type $(0,n-1)$, we have 
\[
\int_D\partial\gamma\wedge\overline{\mathrm{Res}\,\beta}=
\int_D d\gamma\wedge\overline{\mathrm{Res}\,\beta}=
\int_D d(\gamma\wedge\overline{\mathrm{Res}\,\beta})=0,
\]
and the integral is well-defined.

The proof of Theorem \ref{t-3} is done by reducing it to a local
calculation via a partition of unity. It follows from a local version
of the theorem, which we now state, in which we take $\alpha$ to be
smooth with compact support.

Let $X$ be a not necessarily compact complex manifold and $D\subset X$
a smooth hypersurface. Let $\Gamma_c(X,\mathcal A_X^{p,q}(*D))$ be the
vector space of sections of $\mathcal A_X^{p,q}(*D)$ with compact
support. For a cut-off function $\lambda$ and $\epsilon>0$ we consider
the sesquilinear pairing
\[
\langle\ ,\ \rangle_{\lambda,\epsilon}\colon \Gamma_c(X,\mathcal
A^{n,0}_X(*D))\times \Gamma(X,\Omega_X^n(D))\to \mathbb C.
\]
defined by the integral
\[
\langle \alpha,\beta\rangle_{\lambda,\epsilon}=\frac {(-1)^{n-1}} {2\pi
  i}\int_{\lambda\geq\epsilon}\alpha\wedge\bar\beta.
\]
\begin{theorem}\label{t-88} Let $D\subset X$ be a smooth hypersurface in
an arbitrary complex manifold $X$ and $m\in\{0,1,2,\dots\}$.
\begin{enumerate}
\item[(i)] Let
  $\alpha\in \Gamma_c(X,\mathcal A_X^{n,0}(mD))$. Then, as $\epsilon\to0$,
  \[
    \langle\alpha,\beta\rangle_{\lambda,\epsilon} = 2\log\epsilon
    \int_D \mathrm{Res}_\partial\,\alpha\wedge\overline{\mathrm{Res}\,\beta} +
    I_1(\lambda)+O(\epsilon)
  \]
  for some function $I_1(\lambda)$ of the cut-off function.
\item[(ii)] Let $\varphi\in C^\infty(X,\mathbb R)$. Then
  \[
    I_1(e^\varphi\lambda)=I_1(\lambda)- 2\int_D
    \mathrm{Res}_\partial(\varphi\alpha)\wedge\overline{\mathrm{Res}\,\beta},
  \]
 \end{enumerate}
\end{theorem}
Theorem \ref{t-3} follows from Theorem \ref{t-88} via the embedding
$\Omega^n_X(*D)\hookrightarrow\mathcal A^{n,0}(*D)$. We can replace
$\mathrm{Res}_\partial$ by the Leray residue $\mathrm{Res}$ in the
coefficient of $\log\epsilon$ thanks to Prop.~\ref{p-2.6}. We prove Theorem
\ref{t-88} in Section \ref{s-5} by a local calculation involving applying
the Riemann Mapping Theorem to relate the region $\lambda<\epsilon$ to
the region $|f|<\epsilon$ for a local equation $f=0$ of $D$. To do this
we study the asymptotic behaviour of Riemann mappings in the next Section.

%%%%%%%%%%%%%%%%%%%
\section{Asymptotic Riemann mapping}\label{s-4}
Let $\mu(z)=k(z)|z|^2$ where $k(z)>0$ is a positive real analytic
function in a neighbourhood of $0\in\mathbb C$. For small $\epsilon$
the level sets $\mu^{-1}(\epsilon^2)$ bound a simply connected domain
$E_\epsilon$ containing $0$ and by Riemann's mapping theorem there
exists a univalent holomorphic map $h$ from $E_\epsilon$ onto the disk
$|w|<\epsilon$.  The map is unique if we require it to be normalized,
namely that $h(0)=0$ and $h'(0)>0$. The following results give a
description of the behaviour of $h$ as $\epsilon\to 0$.
\begin{theorem}\label{t-1} 
Let $\mu(z)=k(z)|z|^2$ where $k(z)>0$ is a positive real
analytic function  
in a neighbourhood of $0\in\mathbb C$. Then there exist positive constants
$R_1=R_1(\mu),R_2=R_2(\mu)$ and a holomorphic function $h(\epsilon,z)$ defined for
$|\epsilon|<R_1$, $|z|<R_2$ such that for $0<\epsilon<R_1$, $z\mapsto h(\epsilon,z)$
is the normalized univalent map sending the domain $\mu(z)<\epsilon^2$ onto 
the disk $|w|<\epsilon$. Moreover $h$ is an even function of $\epsilon$.
\end{theorem}
\begin{theorem}\label{t-2} The limiting Riemann mapping $h(0,z)$ has the following
description. Let $\mu(z)=\sum_{r,s}c_{r,s}z^{r+1}\bar z^{s+1}$ be the
Taylor expansion of $\mu$. Then
\[
h(0,z)=\sum_{r=0}^\infty \frac{c_{r,0}}{\sqrt{c_{0,0}}}\,z^{r+1}
\]
\end{theorem}  

\begin{example}\label{e-3.1}
Suppose that $\mu(z)=|h(z)|^2$ for some holomorphic function $h(z)$ such
that $h(0)=0$ and $h'(0)>0$. Then $h(\epsilon,z)=h(z)$ is the normalized
univalent map and is independent of $\epsilon$. 
If $h(z)=z(c_0+c_1z+\cdots)$ with $c_0>0$, then the Taylor coefficients
of $\mu(z)$ are $c_{r,s}=c_r\bar c_s$. In particular $c_{r,0}=c_rc_0$ and
$c_{0,0}=c_0^2$.
\end{example}
\begin{example}\label{e-3.2}
Let $\mu(z)=\psi(|z|^2)$ for some real analytic function $\psi$ such
that $\psi(0)=0,\psi'(0)>0$. Then $h(\epsilon,z)=C(\epsilon)z$ where
\[
C(\epsilon)=\frac \epsilon{\sqrt{\psi^{-1}(\epsilon^2)}}.
\]
In this case $c_{0,0}=\psi'(0)$ and $c_{r,0}=0$ for $r>0$, and indeed
$h(0,z)=C(0)z=\sqrt{\psi'(0)}z$. 
\end{example}
\begin{example}\label{e-3.3}
This is rather a counterexample, showing that regularity assumptions
on $\mu$ are necessary.
Consider the family of maps
\[
w=h(\epsilon,z)=\frac {2\epsilon z}{2\epsilon +z}
\]
with parameter $\epsilon>0$. Then $z\mapsto h(\epsilon,z)$ is the
univalent map sending $\mu(z)<\epsilon^2$ to $|w|<\epsilon$, where
$\mu\colon\mathbb C\to \mathbb R_{\geq0}$ is defined by the condition
that $\mu(0)=0$ and $\mu(z)=\epsilon^2$ on the circle
$|z-2\epsilon/3|=4\epsilon/3$ (these circles enclose the origin and
foliate $\mathbb C\smallsetminus 0$).  It can be shown that $\mu(z)$
is smooth except at $z=0$ where it is continuous and obeys
$\mu(z)\leq \text{const}\,|z|^2$. The functions of the family
$h(\epsilon,z)$ does not have a common domain of definition and do not
have a reasonable limit as $\epsilon\to0$.
\end{example}
\begin{example}
Another counterexample, showing that Theorem \ref{t-1} can fail if $\mu(z)$ is
not of the form $k(z)|z|^2$. Let 
\[
\mu(x+iy)=\frac{x^2}{a^2}+\frac{y^2}{b^2}, \quad \text{with $a\neq b$.}
\]
Since $\mu$ is homogeneous the normalized Riemann mapping $h(\epsilon,z)$ sending
$\mu(z)<\epsilon^2$ to $|w|<\epsilon$ obeys
\[
h(\epsilon,z)=\epsilon h(1,z/\epsilon),
\]
Since the Riemann mapping $h(1,z)$ of the ellipse is non-linear, $h(\epsilon,z)$
does not have a limit as $\epsilon\to 0$.
\end{example}
\begin{remark}
The assumption of real analyticity of $\mu$ seems to be essential. It implies that
the boundary of the level sets for small level are real analytic so that---by the
Schwarz reflection principle---the Riemann mappings have an analytic continuation
to a neighbourhood of the boundary, a necessary condition for the existence of
the limiting map, which is not fulfilled for general smooth $\mu$. 
As Example \ref{e-3.3} shows, even the real analyticity of the level sets is not
sufficient for the existence of a limiting Riemann mapping.
\end{remark}
\begin{remark} Let $\mu$ be as in Theorem \ref{t-1}. Let $\psi$ be a real analytic
local diffeomorphism on a neighbourhood of $0\in\mathbb R$ such that $\psi(0)=0$, 
$f$ a local biholomorphic
map such that $f(0)=0$. Then $\tilde\mu=\psi\circ\mu\circ f$ also obeys the assumptions
of Theorem \ref{t-1} and the corresponding normalized Riemann mappings are related by     
\[
\frac{f'(0)}{|f'(0)|}\tilde h(\tilde\epsilon,z)=h(\epsilon,f(z)),\quad \tilde\epsilon^2=\psi(\epsilon^2).
\]
\end{remark}
To prove Theorems \ref{t-1} and \ref{t-2} it is convenient to
consider, instead of the Riemann mapping, the {\em inverse} Riemann
mapping, the univalent map
$w\mapsto z=f_{\epsilon}(w)=w(a_0(\epsilon)+a_1(\epsilon)w+\cdots)$
with $a_0(\epsilon)>0$, sending $|w|^2<\epsilon^2$ onto
$\mu(z)<\epsilon^2$. Also, by rescaling $\epsilon$ we may assume
that $\mu(z)=k(z)|z|^2$ with $k(0)=1$.

The proof is based on the contraction principle: we write the
equation for coefficient the inverse Riemann as a fixed point equation
for a map that is a contraction in a suitable metric space and depends
smoothly on $\epsilon$ including at $\epsilon=0$.

\begin{lemma}\label{l-1} Let
  $\mu(z)=\sum_{r,s=0}^\infty c_{r,s}z^{r+1}\bar z^{s+1}$ be the
  Taylor expansion of $\mu$.  Let $\epsilon>0$ be small. Then
  $z=\sum_{n=0}^\infty a_n(\epsilon)w^{n+1}$ is the normalized
  univalent map sending $D_\epsilon=\{w\in\mathbb C: \ |w|<\epsilon\}$
  onto $E_\epsilon=\{z\in\mathbb C:\mu(z)<\epsilon^2\}$ if and only if the
  series converges to a univalent function on $D_\epsilon$ with
  $a_0(\epsilon)>0$ and the sequence
  $a(\epsilon)=(a_n(\epsilon))_{n=0}^\infty$ obeys
  \[
  a_n(\epsilon)=F_n(\epsilon;a(\epsilon)),\quad n=0,1,2,\dots,
  \]
  where for $n\geq 1$,
  \begin{align*} 
    F_n(\epsilon;a)=&-\sum_{r+s>0} c_{r,s}\sum_{r+|p|-s-|q|=n}\epsilon^{2s+2|q|}
                      a_{p_1}\cdots a_{p_{r+1}}\bar a_{q_1}\cdots\bar a_{q_{s+1}}
    \\
                    &-a_n(\bar a_0-1)-\sum_{k\geq1}a_{n+k}\bar a_k\epsilon^{2k},
  \end{align*}
  with $|p|=\sum_{i=1}^rp_i$ and $|q|=\sum_{i=1}^sq_i$, and for $n=0$,
  \begin{align*}
    F_0(\epsilon;a)
    =1&-\frac12\sum_{r+s>0} c_{r,s}\sum_{r+|p|-s-|q|=0}\epsilon^{2s+2|q|}
        a_{p_1}\cdots a_{p_{r+1}}\bar a_{q_1}\cdots\bar a_{q_{s+1}}
    \\
      &-\frac12|a_0-1|^2-\frac12\sum_{k\geq1}a_{k}\bar a_k\epsilon^{2k}.
  \end{align*}
\end{lemma}
\begin{proof} Since the curve $\mu(z)=\epsilon^2$ is smooth if
  $\epsilon$ is small enough, the Riemann mapping and its inverse
  extend to smooth maps on the closure (in fact to holomorphic maps on
  a neighbourhood of the closure with our assumption of real
  analyticity of $\mu$). Thus the restriction to the boundary of the
  inverse Riemann mapping $f_\epsilon(w)$ has a convergent Fourier
  series and sends the circle $|w|=\epsilon$ to the curve
  $\mu(z)=\epsilon^2$.

  We write this condition as a fixed point equation.  Write
  $f_\epsilon(w)=w(1+g_\epsilon(w))$, $k(z)=1+m(z)$, so that
  \begin{equation}\label{e-m}
    m(z)=\sum_{r+s>0} c_{r,s}z^r\bar z^s.
  \end{equation}
  Then the equation for $g_\epsilon(w)=a_0-1+\sum_{n=1}^\infty a_nw^n$
  is
  \[
    (1+m(f_\epsilon(\epsilon u)))|1+g_\epsilon(\epsilon u)|^2=1,\quad
    |u|=1.
  \]
  We can write this equation in the form
  \begin{equation}\label{e-Laurent}
    g_\epsilon(\epsilon u) + \bar g_\epsilon(\epsilon u^{-1})
    = 
    - m(f_\epsilon(\epsilon u))     |1+g_\epsilon(\epsilon u)|^2 
    - 
    |g_\epsilon(\epsilon u)|^2, \quad |u|=1.
  \end{equation}
  Being an identity of real-valued Fourier series, \eqref{e-Laurent}
  holds if and only if the coefficients of $u^n$ for $n\geq0$ coincide
  on both sides. Let $[\varphi(u)]_n$ denote the $n$th Fourier
  coefficient of a smooth function $\varphi(u)$ on the unit circle
  $|u|=1$.  Together with the condition that $a_0$ is real and
  positive, we get the equivalent condition
  \begin{align*}
    a_0 &= 1-\frac12 \left[ m(f_\epsilon(\epsilon u)) 
          |1+g_\epsilon(\epsilon u)|^2 + 
          |g_\epsilon(\epsilon u)|^2\right]_{0},
    \\
    a_n &= -\frac1{\epsilon^n} 
          \left[m(f_\epsilon(\epsilon u))
          |1+g_\epsilon(\epsilon u)|^2+|g_\epsilon(\epsilon u)|^2\right]_{n}, 
          \quad n\geq1.
  \end{align*}
  This equation has the form $a_n=F_n(\epsilon;a_0,a_1,a_2,\dots)$,
  with
  \begin{align*}
    F_n(\epsilon;a)
    =&-\sum_{r+s>0} c_{r,s}\sum_{r+|p|-s-|q|=n}\epsilon^{r+s+|p|+|q|-n}
       a_{p_1}\cdots a_{p_{r+1}}\bar a_{q_1}\cdots\bar a_{q_{s+1}}
    \\
     &-a_n(\bar a_0-1)-\sum_{k\geq1}a_{n+k}\bar a_k\epsilon^{2k}
    \\
    =&-\sum_{r+s>0} c_{r,s}\sum_{r+|p|-s-|q|=n}\epsilon^{2s+2|q|}
       a_{p_1}\cdots a_{p_{r+1}}\bar a_{q_1}\cdots\bar a_{q_{s+1}}
    \\
     &-a_n(\bar a_0-1)-\sum_{k\geq1}a_{n+k}\bar a_k\epsilon^{2k},
  \end{align*}
  for $n\geq1$. For $n=0$ the calculation is similar.
\end{proof}
We notice that $F(\epsilon;a)$ is defined (a priori only as a formal
power series in the $a_j$) also for $\epsilon=0$.  In this case we can
find the answer explicitly.
\begin{lemma}\label{l-2} 
  Let $\mu(z)=\sum_{r,s=0}^\infty c_{r,s}z^{r+1}\bar z^{s+1}$ with
  $c_{0,0}=1$.  Let
  \[
    h(z)=\sum_{r=0}^\infty c_{r,0}z^{r+1},
  \]
  and $h^{-1}(w)=w(a_0+a_1w+\cdots)$ its inverse. Then
  $a=(a_n)_{n=0}^\infty$ is a solution of $a=F(0;a)$.
\end{lemma}
\begin{proof}
  By Example \ref{e-3.1}, if $\mu(z)=|h(z)|^2$ for some holomorphic
  function $h$ such that $h(0)=0,h'(0)=1$ then $f_\epsilon=h^{-1}$ is
  the univalent map of Lemma \ref{l-1} and is independent of
  $\epsilon$. In this case the Taylor coefficients $c_{r,0}$ are
  precisely the Taylor coefficients of $h$. Thus the equation
  $a=F(0;a)$, which only involves $c_{r,s}$ with $s=0$, is the same
  for a function $\mu$ with arbitrary Taylor coefficients $c_{r,s}$ as
  for $\mu=|h(z)|^2$ with $h(z)=\sum c_{r,0}z^{r+1}$.
\end{proof}

Let $\ell_R^1(\mathbb C)$ be the Banach space of sequences
$a=(a_n)_{n=0}^\infty$ of complex numbers with finite weighted
$\ell^1$-norm $\|a\|_R=\sum_{n=0}^\infty |a_n|R^n$.

\begin{lemma}\label{l-3} 
  Let $0<\delta<1$. With the same notations and assumptions as in
  Lemma \ref{l-1}, there exists an $R>0$ depending on $\mu$ and
  $\delta$, such that the power series $F$ defines a holomorphic map
  \[
    \{\epsilon\in\mathbb C:|\epsilon|<R\}\times B_{\delta}(a^\circ)\to
    B_{\delta}(a^\circ),
  \]
  where
  \[
    B_{\delta}(a^\circ)=B_\delta(a^\circ;\|\ \|_R)=
    \{a\in\ell_R^1(\mathbb C):\|a-a^\circ\|_R<\delta\}
  \] 
  is the open ball of radius $\delta$ centered at
  $a^\circ=(1,0,0,\dots)$ in $\ell_R^1(\mathbb C)$.
\end{lemma}
\begin{proof} We show that $F(\epsilon,a)-a^\circ$ is given by a power
  series in $\epsilon,a_0-1,a_n,\bar a_n$ which converges for
  $|\epsilon|<R$, $\|a-a^\circ\|_R<\delta$, namely that the series
  \[
    \|F(\epsilon;a) - a^\circ\|_R =
    |F_0(\epsilon;a)-1|+\sum_{n=1}^\infty|F_n(\epsilon;a)|R^n.
  \]
  converges absolutely in the specified domain.

  Let $F(\epsilon;a)-a^{\circ}=A+B$ where $A$ is the part involving
  $c_{r,s}$ and $B$ is the quadratic, $c_{r,s}$-independent part.  In
  $|F_0(\epsilon;a)-1|$ we use that $1/2\leq 1$ to obtain a single
  summation for $A$:
  \begin{align*}
    \|A\|_R\leq&\sum_{r+s>0} |c_{r,s}|
                 \sum_{p,q}|\epsilon|^{2s+2|q|}R^{r+|p|-s-|q|}
                 |a_{p_1}\cdots a_{p_{r+1}}\bar a_{q_1}\cdots\bar a_{q_{s+1}}|
    \\
    \leq& \sum_{r+s>0}|c_{r,s}| R^r(|\epsilon|^{2}/R)^s
          \|a\|_R^{r+1}\|a\|_{\epsilon^2/R}^{s+1}
    \\
    \leq& \sum_{r+s>0}|c_{r,s}| R^{r+s}
          \|a\|_R^{r+s+2}.
  \end{align*}
  In the last inequality we use $|\epsilon|\leq R$ so that
  $\|\ \|_{\epsilon^2/R}\leq\|\ \|_{R}$.  To estimate $B$ we set
  $a=a^\circ+b$:
  \begin{align*}
    \|B\|_R
    \leq&\frac12\sum_{k=0}^\infty|b_k|^2|\epsilon|^{2k}
          + \sum_{n=1}^\infty\sum_{k=0}^\infty |b_{n+k}|\, |b_k|R^n|\epsilon|^{2k}
    \\
    \leq&
          \frac12\sum_{k=0}^\infty|b_k|^2R^{2k}
          + \sum_{n=1}^\infty\sum_{k=0}^\infty |b_{n+k}|R^{n+k}\, |b_k|R^k
    \\
    =&\frac12\sum_{k=0}^\infty|b_k|^2R^{2k} 
       + \frac12\sum_{k\neq n} |b_{n}|R^{n}\, |b_k|R^k
    \\
    \leq& \frac12\|b\|_R^2
  \end{align*}
  Therefore, we have the estimate
  \[
    \|F(\epsilon;a)-a^\circ\|\leq \sum_{r+s>0}|c_{r,s}| R^{r+s}
    \|a\|_R^{r+s+2}+\frac12\|a-a^\circ\|_R^2.
  \]
  Assume that $\|a-a^\circ\|_R<\delta$. Then
  $\|a\|_R\leq \|a^\circ\|_R+\delta=1+\delta$ and
  \[
    \|F(\epsilon;a)-a^\circ\| < \sum_{r+s>0}|c_{r,s}|R^{r+s}
    (1+\delta)^{r+s+2} +\delta^2/2.
  \]
  Now choose $R$ so small that
  \[
    \sum_{r+s>0}|c_{r,s}|R^{r+s}(1+\delta)^{r+s+2}\leq\delta-\delta^2/2.
  \]
  Then $\|F(a)-a^\circ\|_R<\delta$.
\end{proof}

\begin{lemma}\label{l-4} With the notation of the previous Lemma,
  $R$ can be chosen so that $a\mapsto F(\epsilon;a)$ is a contraction
  of $B_{\delta}(a^\circ)\subset\ell^1_R(\mathbb C)$ for all
  $\epsilon\in\mathbb C$ such that $|\epsilon<R$.
\end{lemma}

\begin{proof} It is sufficient to show that the differential has norm
  $<1$. Since $F$ is holomorphic we can compute the differential term
  by term. The result is that the differential
  \[
    d_aF(\epsilon;a)(h)=\lim_{t\to
      0}\frac1t(F(\epsilon;a+th)-F(\epsilon;a))
  \]
  is obtained by replacing in each monomial of the power series
  defining $F$ one occurence of $a_p$ or $\bar a_p$ by $h_p$ or
  $\bar h_p$ and multiplying by the degree.  We can then estimate the
  norm in the same way as in the previous lemma.  For
  $h\in \ell^1_R(\mathbb C)$,
  \begin{align*}
    \|d_aF_n(\epsilon;a)(h)\|_R&\leq \sum_{r+s>0}(r+s+2)|c_{r,s}|R^{r+s}
                                 \|a\|_R^{r+s+1}\|h\|_{R}
    \\
                               &+\|h\|_R\|a-a^\circ\|_{R}
    \\
                               &\leq \sum_{r+s>0}(r+s+2)|c_{r,s}|R^{r+s}
                                 (1+\delta)^{r+s+1}\|h\|_{R}
    \\
                               &+\delta\|h\|_R
  \end{align*}
  Since $\mu$ is real analytic, for any given $\delta\in(0,1)$ it is
  possible to choose $R>0$ such that
  \[
    \sum_{r+s>0}(r+s+2)|c_{r,s}|R^{r+s}(1+\delta)^{r+s+1}\leq\theta-\delta,
  \]
  with $\delta<\theta<1$. With this choice of $R$,
  $\|d_aF(\epsilon;a)\|_R\leq\theta$.
\end{proof}

\noindent{\em Proof of Theorem \ref{t-1} and Theorem \ref{t-2}}.  By
the Banach contraction principle there is a unique sequence
$a(\epsilon)\in B_{\delta}(a^\circ)$ obeying the fixed point equation
$a(\epsilon)=F(\epsilon;a(\epsilon))$.  Since
$\sum_{n=0}^\infty|a_n(\epsilon)|R^n<\infty$ and
$|a_0(\epsilon)-1|\leq\|a(\epsilon)-a^\circ\|_R<\delta<1$, the power
series $f_{\epsilon}(w)=w(a_0(\epsilon)+a_1(\epsilon)w+\cdots)$ is
absolutely convergent for $|w|<R$ and defines a univalent map for
$|\epsilon|<R_1$ for sufficiently small $R_1$ ($R_1=R/2$ will do).
Also $a_0(\epsilon)$ is real for real $\epsilon$ since $F_0$ is real;
because $|a_0(\epsilon)-1|<1$, it is positive and thus $f_\epsilon$ is
normalized.  By Lemma \ref{l-1} $f_\epsilon(w)$ is the required
univalent map.  By Lemma \ref{l-3} $F$ is holomorphic, and by Lemma
\ref{l-4} the differential $\mathrm{Id}-d_aF(\epsilon;a(\epsilon))$ at
$a(\epsilon)$ of the map $a\mapsto a-F(\epsilon;a)$ is an automorphism
of the Banach space $\ell^1_R(\mathbb C)$.  By the implicit function
theorem for analytic maps in Banach spaces, see \cite{HajekJohanis}
Ch.~1, Theorem 174, $a(\epsilon)$ is an analytic function of
$\epsilon$ and thus $f_\epsilon(w)$ is analytic as a function of
$\epsilon$ and $w$. It follows that the Riemann mapping
$h(\epsilon,z)=f_\epsilon^{-1}(z)$ is also holomorphic. The fact that
$h$ is an even function of $\epsilon$ follows most simply from the
invariance of the fixed point equation in Lemma \ref{l-1} under
$\epsilon\to-\epsilon$. This proves Theorem \ref{t-1}. Theorem
\ref{t-2} then follows from Lemma \ref{l-2}. \hfill{$\square$}
%%%%%%%%%%%%%%%%%%%%%%%%%%%%%%%%%%%%%%%%%%%%%%%%%%%%%%%%%%%%%%%%%%%%%%%%%%%%%%

\section{Proof of Theorem \ref{t-88}}\label{s-5}
By using a partition of unity we may assume that $X$ is an open
neighbourhood of $0\in\mathbb C^n$ and that $D$ is defined by the
equation $z_1=0$ where $z_1$ is the first coordinate function on
$\mathbb C^n$. The main technical step is to settle the case where
$\alpha=\partial\gamma$ is exact:
\begin{lemma}\label{l-5.1}
Let $m\in\{0,1,2,\dots\}$. If $\alpha=\partial\gamma$ with
  $\gamma\in \Gamma_c(X,\mathcal A_X^{n-1,0}(mD))$ then, as
  $\epsilon\to 0$,
  \[
    \langle\alpha,\beta\rangle_{\lambda,\epsilon} = 2\int_D
    \mathrm{Res}_\partial(\partial\log\lambda\wedge\gamma) \wedge
    \overline{\mathrm{Res}\,\beta}+O(\epsilon)
  \]
\end{lemma}
\begin{proof}  We first reduce the general case to the case where
$\lambda^2$ is real analytic, and in fact a polynomial. Let $\lambda$
be a cut-off function and $\lambda_N$ be the cut-off function such
that  $\lambda_N^2$ is the degree $N$ Taylor polynomial in the variable
${z_1}$ at ${z_1}=0$. We show that for fixed $\alpha,\beta$ and $N$ large enough,
\[
\langle\alpha,\beta\rangle_{\lambda,\epsilon}=
\langle\alpha,\beta\rangle_{\lambda_N,\epsilon}+O(\epsilon)
\]
For small $\epsilon$, 
$U_\epsilon=\lambda^{-1}([0,\epsilon))\cap\mathrm{supp}\,\alpha$ 
is contained in the region $|z_1|<C_1\epsilon$ for some $C_1>0$
and thus
$|\lambda_N(z)-\lambda(z)|<C_2\epsilon^{N+1}$ for all $z\in U_{\epsilon}$ and
some $C_2>0$. 
It follows that the level set $\lambda_N^{-1}(\epsilon)$ is contained
in $\lambda^{-1}([\epsilon-C\epsilon^{N+1},\epsilon+C\epsilon^{N+1}]$.
The difference between the integrals over $\lambda(z)>\epsilon$ and
$\lambda_N(z)>0$ is then estimated by the volume of the region,
which is less than $\mathrm{const}\,\epsilon^{N}$ times the maximum
of the integrand, which is less than $\mathrm{const}\,\,\epsilon^{-m-1}$
if $\alpha$ has a pole of order $m$. For $N>m-1$ the difference
vanishes in the limit $\epsilon\to  0$.

So we assume from now on that $\lambda^2$ is a real analytic function
of $z$.

 By Stokes's theorem,
\begin{align*}
\langle\partial\gamma,\beta\rangle_{\lambda,\epsilon}
&=\frac{(-1)^{n}}{2\pi i}\int_{\lambda=\epsilon}\gamma\wedge\bar\beta
\end{align*}
We prove the identity in several steps.
\subsection*{(a) The one-dimensional case with $\lambda(z)=|z|$.} 
Suppose $n=1$ and write $z_1=z$. 
Then $\gamma(z)=z^{-m}\sigma(z)$ for some smooth
function $\sigma$ and $\bar\beta=\bar b(\bar z)d\bar z/\bar z$ for
some holomorphic function $\bar b$ defined around $0$. Let us first
assume that $\lambda(z)=|z|$ and let
$\sigma(z)\sim\sum_{r,s=0}^\infty\sigma_{rs}z^r\bar z^s$ be the Taylor
series of $\sigma$. Then
\[
  \int_{|z|=\epsilon} \gamma\wedge\bar\beta = -\sum_{r+s\leq N}
  \sigma_{r,s}\epsilon^{r+s-m}\int_{|u|=1}u^{r-s-m}\bar b(\epsilon
  u^{-1})\frac{du}{u}+O(\epsilon),
\]
for $N$ sufficiently large, say $N\geq m+1$. Thus we have to evaluate
\[
  \epsilon^{r+s-m}\int_{|u|=1}u^{r-s-m}(\epsilon u^{-1})^l\frac{du}u
  =2\pi i \delta_{r,s+m+l}\epsilon^{2s+2l},
\]
for $r,s,m,l\geq0$. Thus only one term in the sum and only the leading
term $\bar b(0)$ survive the limit $\epsilon\to0$.  Therefore
\[
  \int_{|z|=\epsilon} \gamma\wedge\bar\beta=-2\pi
  i\,\sigma_{m,0}\,\bar b(0)+O(\epsilon).
\]
On the other hand, $b(0)=\mathrm{Res}\,\beta$ and
\[
  \mathrm{Res}_\partial\left(\partial\log\lambda\,\gamma\right)
  =\mathrm{Res}_\partial\,\frac{dz}z\gamma(z) =\sigma_{m,0},
\]
as computed in Example \ref{e-2.1}. The proof is complete.

\subsection*{(b) The one-dimensional case with general $\lambda$.}
Let again $n=1$, but with a general cut-off function $\lambda$.  By
Theorem \ref{t-1}, for sufficiently small $\epsilon$, the normalized
biholomorphic Riemann map $z=g_\epsilon(w)$ sending the disk
$|w|<\epsilon$ to the domain $\lambda(z)<\epsilon$ extends to a smooth
parametrization of the boundary and has an analytic continuation to a
holomorphic function of $\epsilon,z$ defined on a polydisk around
zero. We change variables to reduce to case (a):
\[
  \int_{\lambda(z)=\epsilon}\gamma\wedge\bar\beta =
  \int_{|w|=\epsilon}g_\epsilon^*\gamma
  \wedge\overline{g_\epsilon^*\beta}.
\]
In this formula we may replace $g_\epsilon$ by its limit for
$\epsilon=0$, since the difference is $\epsilon$ times an integral of
the form considered in (a) which has a finite limit as
$\epsilon\to 0$.  The required identity
\[
  \frac{-1}{2\pi i}\int_{\lambda(z)=\epsilon}\gamma\wedge\bar\beta
  =\mathrm{Res}_\partial(\partial\log\lambda\,\gamma)\wedge
  \overline{\mathrm{Res}\,\beta}+O(\epsilon)
\]
thus holds in the coordinate $w=g_0(z)$ and thus in any coordinate system.
\subsection*{(c) The case of arbitrary dimension I} Now let $n$ be
arbitrary and suppose that
\[
  \gamma= f dz_2\wedge\cdots\wedge dz_n
\]
for some function $f$.  
Then the integration over each fibre of the projection
$z\to (z_2,\dots,z_n)$ is of the form considered in
(b). The claim then follows from (b), (the sign comes from permuting
$d\bar z_1$ through $dz_2\wedge\cdots\wedge dz_n$.)

\subsection*{(d) The case of arbitrary dimension II}
Now we assume that 
\[
  \gamma=\frac {d{z_1}}{{z_1}}\wedge\sigma
\]
for some smooth $(n-2)$-form $\sigma$. On the cycle $\lambda=\epsilon$
we have the identity
\[
  \partial\log\lambda+\bar\partial\log\lambda=0.
\]
Let us again introduce coordinates $z_1,z_2,\dots,z_n$ and decompose
the differential as $\partial=d{z_1}\partial_{z_1}+\partial'$. The identity becomes
\[
  d{z_1}\,\partial_{z_1}\log\lambda+\partial'\log\lambda \equiv 0\mod \langle
  d\bar z_1,\dots,d\bar z_n\rangle
\]
modulo the submodule generated by $d\bar z_i$, which doesn't contribute when
we multiply by the $(0,n)$-form $\bar \beta$. 
Thus
\[
  \int_{\lambda=\epsilon}\gamma\wedge\bar\beta =
  -\int_{\lambda=\epsilon}\frac{\partial'\log
    \lambda}{{z_1}\,\partial_{z_1}\log\lambda} \wedge\sigma\wedge\bar\beta.
\]
The denominator ${z_1}\partial_{z_1}\log\lambda$ is actually smooth and
non-zero on $D$, since it can be written as $1+{z_1}\partial_{z_1}\log h$,
where $\lambda=h|{z_1}|$ with smooth nonzero $h$.

Now the integral is of the form considered in (c) for which we have
proved the claim. Thus
\[
  \frac{(-1)^{n}}{2\pi
    i}\int_{\lambda=\epsilon}\gamma\wedge\bar\beta=-
  \int_D\mathrm{Res}_\partial\left(\partial\log\lambda\wedge\frac{\partial'\log
      \lambda}{{z_1}\partial_{z_1}\log\lambda}
    \wedge\sigma\right)\wedge\overline{\mathrm{Res}\,\beta}+O(\epsilon).
\]
Since $\partial'\log\lambda\wedge\sigma$ is an $(n-1,0)$-form not involving
$d{z_1}$, we may replace $\partial\log\lambda$ by $d{z_1}\partial_{z_1}\log\lambda$ on the
right-hand side and we obtain
\begin{align*}
  \frac {(-1)^{n}} {2\pi i}
  \int_{\lambda=\epsilon} 
  \gamma \wedge \bar\beta&=-
                           \int_D\mathrm{Res}_\partial \left( \frac{d{z_1}}{{z_1}} 
                           \wedge
                           \partial'\log \lambda
                           \wedge
                           \sigma\right)
                           \wedge
                           \overline{\mathrm{Res}\,\beta}+O(\epsilon)
  \\
                         &=-
                           \int_D\mathrm{Res}_\partial 
                           \left(
                           \frac{d{z_1}}{{z_1}}
                           \wedge
                           \partial\log \lambda
                           \wedge
                           \sigma\right)
                           \wedge
                           \overline{\mathrm{Res}\,\beta}+O(\epsilon)
  \\
                         &=
                           \int_D\mathrm{Res}_\partial
                           \left(
                           \partial\log \lambda
                           \wedge
                           \gamma\right)
                           \wedge
                           \overline{\mathrm{Res}\,\beta}+O(\epsilon),
\end{align*}
so that the claim also holds in this case.

\subsection*{(e) The case of arbitrary dimension III}
We finally deal with the most general case. The form $\gamma$ may
be written as
\[
  \gamma=fdz_2\wedge\cdots\wedge dz_n+{dz_1}\wedge\sigma
\]
for some function $f$ and an $(n-2)$-form $\sigma$ such that $z_1^mf$
and $z_1^m\sigma$ are smooth.  Since we can add a $\partial$-closed
form to $\gamma$ without changing either side of the claim, we may by
Prop.~\ref{p-1} assume that $\sigma$ has a first order pole. Indeed if
$\sigma$ has a pole of order $p+1>1$, say $\sigma=z_1^{-p-1}\tau$,
then
$dz_1\wedge\sigma=\frac1pz_1^{-p}\partial\tau-\partial(\frac1pz_1^{-p}\,\tau)$. Thus
we can lower the order of the pole by adding exact forms until
$\sigma$ has a pole of first order.  In this way, the general case is
reduced to the cases (c) and (d).
\end{proof}

\noindent{\it Proof of Theorem \ref{t-88}.}
(i) By Lemma \ref{l-5.1} and Prop.~\ref{p-1}, it is sufficient to
consider the case where $\alpha$ has a simple pole. Write
\[
  \alpha=f\,\frac{dz_1}{z_1}\wedge dz_2\wedge\cdots\wedge dz_n,\quad
  \beta=g\,\frac{dz_1}{z_1}\wedge dz_2\wedge\cdots\wedge dz_n,
\]
where $f$ is a function with compact support and $g$ is holomorphic.
We may assume that the support of $f$ is contained in the region
$|z_1|<1$. In the integral of $\alpha\wedge\bar \beta$ we may
replace $f$ and $\bar g$ by their values at $z_1=0$, as the difference is a
regularized absolutely convergent integral, whose limit as $\epsilon\to0$
is independent of the choice of $\lambda$. We obtain
\[
  \int_{|z_1|\leq1,\lambda\geq\epsilon}\alpha\wedge\bar\beta=(-1)^{n-1}
  \int_D\int_{|z_1|\leq1,\lambda\geq\epsilon}\frac {dz_1\wedge d\bar z_1}
  {z_1\bar z_1}\wedge \mathrm{Res}\,\alpha\wedge\overline{\mathrm{Res}\,\beta}+\cdots,
\]
where the dots denote a term whose limit as $\epsilon\to 0$ exists and is
independent of $\lambda$. We are left with a two-dimensional integral over
$z_1$
which we evaluate in polar coordinates: 
let $r=r_\epsilon(\theta)=\epsilon/h(0)(1+O(\epsilon))$ the polar
parametrization of the curve $\lambda(z_1)=\epsilon$ with $\lambda(z_1)=|z_1|h(z_1)$. Then
\begin{align*}
\int_{|z_1|\leq1,\lambda\geq\epsilon}\frac {dz_1\wedge d\bar z_1}
  {z_1\bar z_1}&=-2i\int_0^{2\pi}\int_{r_\epsilon(\theta)}^1\frac{dr}rd\theta
\\
&= 2i\int_{0}^{2\pi}\log\,r_\epsilon(\theta)\,d\theta
\\
&=4\pi i\log(\epsilon/h(0))+O(\epsilon). 
\end{align*}
It follows that 
\begin{equation}\label{e-99}
\langle\alpha,\beta\rangle_{\lambda,\epsilon}=2\,
 \int_D\log\left(\frac\epsilon{h|_D}\right)\mathrm{Res}\,\alpha\wedge\overline{\mathrm{Res}\,\beta}+\cdots,\qquad \lambda=h|z_1|,
\end{equation}
up to a finite term whose limit as $\epsilon\to 0$ is independent of $\lambda$. Since $\mathrm{Res}$ coincides with $\mathrm{Res}_\partial$ for $(n,0)$-forms
with simple pole, the proof is complete.

(ii) If $\alpha$ has a first order pole, the claim follows from
\eqref{e-99}: if we multiply $\lambda$ by $\exp\varphi$, then
$\log\,h|_D$ changes by $\varphi|_D$, and
$\varphi\,\mathrm{Res}_\partial\,\alpha=\mathrm{Res}_\partial(\varphi\alpha)$
in the case of first order pole.  By Prop.~\ref{p-1}, a general form
$\alpha$ is a the sum of a $\partial$-exact form and a logarithmic
form. It is thus sufficient to check the claim when
$\alpha=\partial\gamma$ is $\partial$-exact, for which we use Lemma
\ref{l-5.1}. Since $\mathrm{Res}_\partial$ vanishes on
$\mathrm{Im}\,\partial$, we have
\begin{align*}
  \langle\partial\gamma,\beta\rangle_{e^\varphi\lambda,\epsilon}
  -\langle\partial\gamma,\beta\rangle_{\lambda,\epsilon} 
  &=
    2\int_D\mathrm{Res}_\partial(\partial\varphi\wedge\gamma)
    \wedge\overline{\mathrm{Res}\,\beta}+O(\epsilon)
  \\
  &=
    -2\int_D\mathrm{Res}_\partial(\varphi\,\partial\gamma)
    \wedge\overline{\mathrm{Res}\,\beta}+O(\epsilon).
\end{align*}
Thus the claim also holds for $\partial$-exact $\alpha$, completing
the proof
\hfill $\square$

\end{document}